\documentclass[12pt]{article}

\usepackage[centertags]{amsmath}
\usepackage{amsfonts}
\usepackage{amssymb}
\usepackage{amsthm}
\usepackage{graphicx,caption}
\usepackage{verbatim}

\newtheorem{theo}{Theorem}
\newtheorem{lemma}[theo]{Lemma}
\newtheorem{cor}[theo]{Corollary}
\newtheorem{prop}[theo]{Proposition}

\newtheorem{example}[theo]{Example}

\newcommand\be{\begin{equation}}
\newcommand\ee{\end{equation}}

\allowdisplaybreaks

\begin{document}
\title{Planar self-affine sets with equal Hausdorff, box and affinity dimensions}
\author{Kenneth Falconer \& Tom Kempton}
\maketitle
\abstract{\noindent  Using methods from ergodic theory along with properties of the Furstenberg measure we obtain conditions under which certain classes of plane self-affine sets have Hausdorff or box-counting dimensions equal to their affinity dimension. We exhibit some new specific classes of self-affine sets for which these dimensions are equal.
\footnote{This research was supported by EPSRC grant EP/K029061/1}}

\section{Introduction}
\setcounter{equation}{0}
\setcounter{theo}{0}
A family of contractive maps  $\{T_1,\ldots,T_m\}$ on  $\mathbb R^n$ is termed an {\it iterated function system} or IFS. By standard IFS theory  \cite{Fal,Hut}  there exists a non-empty compact subset of $\mathbb R^n$ satisfying 
\begin{equation}\label{ifs}
E=\bigcup_{i=1}^m T_i(E),
\end{equation}
called  the {\it attractor} of the IFS. 
If the $T_i$ are affine transformations, that is of the form
\be\label{affdef}
T_i x = A_i x+d_i \qquad (1 \leq i \leq m)
\ee
where $A_i$ are linear mappings or matrices on $\mathbb R^n$ with $||A_i||_2<1$ and $d_i \in \mathbb{R}^2$ are translation vectors,  $E$ is termed  a {\it self-affine} set. In the special case when the $T_i$ are all similarities $E$ is called {\it self-similar}.
Self-affine sets are generally fractal and it is natural to investigate their Hausdorff and box-counting dimensions. Whilst the dimension theory  is well-understood in the special case of self-similar sets, at least assuming  some separation or disjointedness condition for the union in  \eqref{ifs}, see \cite{Fal,Hut}, dimensions of self-affine sets are more elusive, not least because the dimensions do not everywhere vary continuously in their defining parameters \cite{Fal}.

The affinity dimension $\dim_A E$ of a self-affine set $E$, which  is defined in terms of the linear components $A_i$ of the affine maps, see \eqref{dimaff}, turns out to be central to these studying the dimensions of self-affine sets.  It is always the case that 
$$\dim_H E \leq {\underline\dim}_B E \leq{\overline\dim}_B E \leq\dim_A E,$$
where ${\underline\dim}_B, {\overline\dim}_B$ and $\dim_H$ denote lower and upper box-counting and Hausdorff  dimensions, see \cite{Fal,Mat} for the definitions. However, in  many situations equality holds here  `generically', that is for almost all parameters in a parametrized family of self-affine sets, see for example \cite{Falconer88,FalSur,Solomyak98}. However, in general, it is not easy to identify for which parameters the generic conclusion holds.

Exact values of Hausdorff and/or box dimensions have been found for several classes of self-affine sets, see the survey \cite{FalSur} and references therein. Particular attention has been given to `carpets' where the affinities preserve horizontal and vertical directions, see \cite{Bedford, FraserBoxlike, Hu98, McMullen}. Such examples are often exceptions to  the generic situation: the box and Hausdorff dimensions need not be equal nor need they equal the affinity dimension. 

By pulling back elongated images of a self-affine set under compositions of affine mappings, it is easy to see that the small scale coverings needed for estimating Hausdorff and box dimensions are related to the projections of the set in certain directions. Indeed, in the case of carpets, as in \cite{Bedford, FraserBoxlike, Hu98, McMullen}, these dimensions depend on the projection of the sets, or projections of measures supported by the set, onto the weak contracting direction, see also \cite{Hu98} for a generalisation of this to other constructions where there is a weak contracting foilation. 

However, self-affine sets do not in general have an invariant contracting direction.  The appropriate analogue is to examine the typical dimension of the projection in directions chosen according to the Furstenberg measure $\mu_F$ on the projective line $\mathbb{RP}^1$ which is supported by the relevant set of directions.  The Furstenberg measure $\mu_F$ is induced in a natural way by the K\"aenmaki measure $\mu$, \cite{Kaenmaki}, which is supported by $E$ and which typically has Hausdorff dimension $\dim_H\mu  = \dim_H E$, see Section \ref{secfur}. 

Throughout this paper $E$ will be  a self-affine subset of $\mathbb R^2$ which satisfies the {\it strong separation condition}, that is with the union in \eqref{ifs} disjoint, and such that the linear parts of the defining the affine transformations map the first quadrant into itself, corresponding to the $A_i$ having strictly positive entries. Our two main theorems relate to sets $E$ with dimension at least 1. The first gives conditons for the (lower) Hausdorff dimension of $E$ to equal its affinity dimension, and this depends on the absolute continuity of the projections of the measure $\mu$. 
By contrast, the second theorem, which gives  conditions for equality of the box-counting dimension and affinity dimension of $E$, depends on the projections of the set $E$ itself. This dichotomy is analogous to that with Bedford-McMullen carpets \cite{Bedford,McMullen} where the Hausdorff and box dimensions of the carpets may be expressed in terms of the projection in the unique contracting direction of measures and sets respectively.

\begin{theo}\label{HDThm}
Let $E\subset \mathbb R^2$ be the  self-affine set defined by the IFS \eqref{affdef} where the $A_i$  are strictly positive matrices and the strong separation condition is satisfied. Let $\mu$ be the K\"aenm\"aki measure and $\mu_F$ the corresponding Furstenberg measure. Suppose that for $\mu_F$-almost all $\theta$ the projection of $\mu$ in direction $\theta$ is absolutely continuous. Then the measure $\mu$ is exact dimensional  and $\dim_H E= \dim_B E=\dim_A E$.
\end{theo}

\begin{theo}\label{BoxThm}
Let $E\subset \mathbb R^2$ be the  self-affine set defined by the IFS \eqref{affdef} where the $A_i$  are strictly positive matrices and the strong separation condition is satisfied. Suppose that the projection of $E$ has positive Lebesgue measure in a set of directions of positive $\mu_F$-measure. 
Then $\dim_B E= \dim_A E$.
\end{theo}

Note that the conclusion of Theorem \ref{BoxThm} was obtained in \cite{FalSA2} under the much stronger condition of the  projection of $E$ in all directions  having Lebesgue measure greater than some positive constant.

A number of corollaries follow easily from these theorems.

\begin{cor}\label{Cor1}
Let $E\subset \mathbb R^2$ be the  self-affine set defined by the IFS \eqref{affdef} where the $A_i$  are strictly positive matrices and the strong separation condition is satisfied. Let $\mu$ be the corresponding K\"aenmaki measure and  assume that $\dim_H \mu >1$.
If the Furstenberg measure $\mu_F$  is absolutely continuous with respect to Lebesgue measure on $\mathbb{RP}^1$ then $\dim_H E= \dim_B E = \dim_A E$.
\end{cor}
\begin{proof}
Since $\dim_H \mu>1$, Marstrand's projection theorem \cite{Fal,HuntKaloshin,Mar} implies that the projection of $\mu$  is absolutely continuous in Lebesgue-almost every direction, and hence in $\mu_F$-almost every direction, since $\mu_F$ is absolutely continuous. The conclusion follows from Theorem  \ref{HDThm}.
\end{proof}

Note that for large regions of parameter space, for almost every collection of matrices $\{A_1\cdots A_m\}$ the corresponding Furstenberg measure is absolutely continuous \cite{BPS}, in which case Corollary \ref{Cor1} applies. 

The next corollary often enables us to obtain the precise value of $\dim_H E$  by just finding rather crude lower bounds for $\dim_H E$ and $\dim_H \mu_F$.

\begin{cor}\label{Cor2}
Let $E\subset \mathbb R^2$ be the  self-affine set defined by the IFS \eqref{affdef} where the $A_i$  are strictly positive matrices and the strong separation condition is satisfied. Let $\mu$ be the corresponding K\"aenmaki measure.
Suppose that $\dim_H \mu+ \dim_H \mu_F>2$. Then  $\dim_H E= \dim_B E = \dim_A E$.
\end{cor}

\begin{proof}
Since $\dim_H \mu_F \leq 1$ the assumption requires that $\dim_H \mu>1$. By results on the dimension of the exceptional set of projections \cite{Fal2,Mat}, the projection of $\mu$ in direction $\theta$ is absolutely continuous
for all $\theta$ except for a set of $\theta$ of Hausdorff dimension at most $2-\dim_H \mu< \dim_H \mu_F$. Hence the projection of $\mu$ is absolutely continuous in $\mu_F$-almost all directions, so the conclusion follows from Theorem  \ref{HDThm}.
\end{proof}

In Section 5 we use these ideas to give explicit constructions of classes of self-affine sets which have Hausdorff dimensions equal to their affinity dimensions. 

When writing up this research the authors became aware of a preprint \cite{Bar} which also gives an ergodic theoretic approach to  self-affine sets and measures, though the methods and specific examples there are very different.

\section{Preliminaries}
\setcounter{equation}{0}
\setcounter{theo}{0}

After rescaling, which does not affect dimension, we may assume that each $T_i$  in \eqref{affdef} maps the unit disk $D$ strictly inside itself.
We denote composition of functions by concatenation and write $T_{a_1\cdots a_n} = T_{a_1}T_{a_2}\cdots T_{a_n}$, etc, where $1 \leq a_i \leq m$. Similarly, we write $E_{a_1\cdots a_n} = T_{a_1\cdots a_n}(E)$ for the image of $E$ under such compositions.
Let $\alpha_1(a_1\cdots a_n)\geq \alpha_2(a_1\cdots a_n)>0$ be the {\it singular values} of $A_{a_1\cdots a_n}$, that is  the lengths of the major and minor semiaxes of the  ellipses
$D_{a_1\cdots a_n}$, or equivalently the positive square roots of the eigenvalues of $A_{a_1\cdots a_n}A^T_{a_1\cdots a_n}$.
Note that $\alpha_1$ and $ \alpha_2$ depend only on the $A_{i}$ and are independent of the translations $d_i$.
The {\it affinity dimension}  a set of linear mappings on $\mathbb R^2$ or $2\times 2$ matrices is given by
\be\label{dimaff}
\dim_{A}(A_1,\ldots,A_m)=\inf\Big\{s:\sum_{n=1}^{\infty}\sum_{a_1\cdots a_n\in\{1,\ldots,m\}^n}\phi^s(A_{a_1\cdots a_{n}})<\infty\Big\},
\ee
where
\[
\phi^s(A)=\left\lbrace\begin{array}{cc}\alpha_1^s & 0<s\leq 1\\
                       \alpha_1\alpha_2^{s-1} & 1\leq s \end{array}\right. ,
\]
for a matrix $A$ with singular values $\alpha_1\geq \alpha_2>0$, see \cite{Falconer88, Fal}.
When the transformations that define a self-affine set $E$ are clear, we often write $\dim_A E$ for the affinity dimension, though strictly it depends on the defining IFS of $E$. 
We seek conditions under which the Hausdorff dimension  or box dimension of a self-affine set coincides with its affinity dimension. 

We set $\Sigma := \{1,\ldots,m\}^{\mathbb N}$ and for the infinite word $a = a_1 a_2 \cdots\in \Sigma$ we write $a|_n := a_1\cdots a_n$ for its curtailment after $n$ letters. Subsets of $\Sigma$ of the form $[b] := \{a\in \Sigma: a|_n=b\}$, where  $b \in \{1,\ldots,m\}^n$ is a finite word, are called {\it cylinders}. Since each $T_i$ is a contraction, for each ${a}\in \Sigma$ and $y\in \mathbb R^2$ the sequence $(T_{a|_n}(y))$ has a unique limit point $x\in E$ which is independent of the choice of $y\in \mathbb R^2$. We call the word ${a}= a_1a_2\cdots$ the {\it code} of $x$, and define the {\it projection} $\pi:\Sigma\to E$ to be the map $\pi({a}) = \lim_{n \to \infty} T_{a|_n}(y)$;  the strong separation condition implies that $\pi$ is a bijection.

Let $\mu$ be a measure on $ \Sigma$ which we identify with a measure on subsets of $E$ under $\pi$ in the natural way, so that $\mu[a|_n] = \mu(E_{a|n})$. The {\it Lyapunov exponents} $\lambda_1(\mu), \lambda_2(\mu)$ are defined as the constants such that, for $\mu$-almost every ${a}\in \Sigma$,
\begin{equation}\label{lyap}
\lim_{n\to\infty}\frac{1}{n}\log \alpha_i(a|_n)=\lambda_i
\end{equation}
for $i\in\{1,2\}$. The {\it Lyapunov dimension} of $\mu$ is given by
\begin{equation}\label{lyapdim}
D(\mu):= \left\lbrace
\begin{array}{cc}
\dfrac{h(\mu)}{-\lambda_1(\mu)} & \quad h(\mu)\leq -\lambda_1(\mu)\\
1+\dfrac{h(\mu) + \lambda_1(\mu)}{-\lambda_2(\mu)}& \quad h(\mu)\geq -\lambda_1(\mu)
\end{array}\right. ,
\end{equation}
where $h(\mu)$ is the Kolmogorov-Sinai entropy of the system $(\Sigma,\sigma,\mu)$ and $\sigma$ is the left shift on $\Sigma$.
Note that $D(\mu)$ depends only on the matrices $\{A_1,\cdots, A_m\}$ and the measure $\mu$. 

There exists a probability measure $\mu$ on $\Sigma$, known as the {\em K\"aenm\"aki measure}, which is ergodic and shift invariant and satisfies $D(\mu)=\dim_A(A_1,\cdots,A_m)$, see \cite{Kaenmaki} and \cite[Proposition 1.8]{JPS}. Furthermore, from  \cite[Theorem 3.5]{KaenmakiReeve}, $\mu$ is a Gibbs measure assuming, as we do,  that the matrices $A_i$ are strictly positive\footnote{Indeed, while this paper was under review, it was shown by B\'ar\'any and Rams \cite{BaranyRams} that $\mu$ is a Gibbs measure associated to an additive, rather than just a subadditive potential.}. From now on $\mu$ will denote this probability measure. We will prove  in the setting of Theorem \ref{HDThm} that  $\dim_H \mu =D(\mu)$ from which equality of the $\dim_H E$ with the affinity dimension follows.

Here we define
\[
\dim_H \mu:=\inf\{\dim_H A :\mu(A)=1\}.
\]
This is sometimes known as lower Hausdorff dimension, although all notions of Hausdorff dimension coincide for exact dimensional measures, and we will see later that the measures that we use are exact dimensional.

\section{The Furstenberg measure and dynamics on projections}\label{secfur}
\setcounter{equation}{0}
\setcounter{theo}{0}

For $1\leq i \leq m$ let $\phi_i:\mathbb{RP}^1 \to \mathbb{RP}^1$ be the projective linear transformation on the projective line  associated with the matrix $A_i^{-1}$ given by
\begin{equation}\label{phidef}
\phi_i(\theta)=\dfrac{A_i^{-1}(\theta)}{\|A_i^{-1}(\theta)\|}
\end{equation}
where $\|\ \|$ denotes the Euclidean norm, and where we parameterize $\mathbb{RP}^1$ by unit vectors in the obvious way.
The {\it Furstenberg measure} $\mu_F$ is defined to be the stationary measure on $\mathbb{RP}^1$ associated to the maps $\phi_i$ chosen according to the measure $\mu$. Alternatively, setting
\[
\phi_{a_n\cdots a_1}:=\phi_{a_n}\phi_{a_{n-1}}\cdots\phi_{a_1},
\]
then for $\mu$-almost every ${a}\in\Sigma$ and all $\theta\in \mathbb{RP}^1$, the sequence of measures
\[
\frac{1}{n}\sum_{k=0}^{n-1} \delta_{\phi_{a_k\cdots a_1}(\theta)}
\]
converges weakly to  $\mu_F$ on $\mathbb{RP}^1$. See B{\'a}r{\'a}ny, Pollicott and Simon \cite{BPS} for further discussion of the Furstenberg measure. 

With strictly positive matrices $A_i$, the transformations $\phi_i$ are strict contractions of the negative quadrant $\mathcal Q_2\subset \mathbb{RP}^1$ under a metric $d(\theta_1,\theta_2)$ given by the absoute angle between $\theta_1,\theta_2 \in \mathcal Q_2$. With respect to this metric, the Furstenberg measure is an invariant probability measure on the strictly contractive IFS $\{\phi_1,\ldots,\phi_m\}$. Alternatively one could work with the variant of the Hilbert metric $d_H$  discussed by Birkhoff \cite{Birkhoff}.

For $\theta \in  \mathbb{RP}^1$ let $\pi_{\theta}:E\to[-1,1]$ be the map obtained by projecting the self-affine set $E$ defined by \eqref{ifs} in direction $\theta$ onto the diameter of the unit disc $D$ at angle $\theta^{\perp}$, and then isometrically mapping this diameter onto $[-1,1]$. With a slight abuse of notation, we also denote by $\pi_{\theta}$ the composition $\pi_{\theta}\circ \pi: \Sigma \to [-1,1]$.
For each $\theta$ let $\mu_{\theta}:=\mu\circ\pi_{\theta}^{-1}$ be the corresponding projection of the measure $\mu$ onto $[-1,1]$.

We now consider the two-sided shift space $\Sigma^{\pm}:=\{1,\cdots,m\}^{\mathbb Z}$ where we denote a typical member $\cdots a_{-2}a_{-1}a_0a_1a_2\cdots$ by $\overline{a}$. We define  $\overline{\mu}$ to be the unique, shift invariant measure  on $\Sigma^{\pm}$ for which $\mu[a_m\cdots a_n]=\overline{\mu}[a_m\cdots a_n]$ for every cylinder depending only on positive coordinates. 

The limit
\[
\rho(\overline{a}):=\lim_{n\to\infty} \phi_{a_0}\phi_{a_{-1}}\cdots \phi_{a_{-n}}(\theta)
\]
exists for all $\overline{a}\in\Sigma^{\pm}$ and is independent of $\theta\in\mathcal Q_2$, this is just a standard iterated function system argument since the maps $\phi_i$ contract $\mathcal Q_2$.

We define the map $P:\Sigma^{\pm}\to \mathbb{RP}^1\times\Sigma$ by
\[
P({\overline{a}}):= \big(\rho(\overline{a}), a_1a_2\cdots\big)
\]
Here the non-positive coordinates of $\overline{a}$ determine an angle in $\mathcal Q_2$ and the positive coordinates are unchanged. 

Let $\nu$ be the measure on $\mathcal Q_2\times\Sigma$ defined by pushing forward $\overline{\mu}$ by $P$, formally
\[
\nu:=\overline{\mu}\circ P^{-1}.
\]

\begin{lemma}\label{ergodic}
The map $P\circ \sigma\circ P^{-1}:\mathbb{RP}^1\times\Sigma\to\mathbb{RP}^1\times\Sigma$ is well defined. Furthermore, the system $(\mathbb{RP}^1\times\Sigma,\nu,P\circ\sigma\circ P^{-1})$ is ergodic.
\end{lemma}

Note that while $P^{-1}(\theta,\underline a)$ may be set valued, $P\circ\sigma\circ P^{-1}$ is well defined, since if $P(\overline a)=P({\overline a}')$ for ${\overline a},{\overline a}' \in \Sigma^{\pm}$ then
$P\circ\sigma(\overline a)=P\circ\sigma({\overline a}')$.

\begin{proof}
Given $(\theta,a)$ $\in\mathbb{PR}^1\times\Sigma$, the set $P^{-1}(\theta,a)$ consists of those two-sided sequences $\overline a$ for which $\rho(\overline a)=\theta$. Then $\rho(\sigma(\overline a))=\phi_{a_1}(\theta)$ and 
\[
P\circ \sigma\circ P^{-1}(\theta,a)=(\phi_{a_1}(\theta), \sigma(a)).
\]
Now $(\mathbb{RP}^1\times\Sigma,\nu,P\circ\sigma\circ P^{-1})$ is a factor (under the map $P$) of the ergodic system $(\Sigma^{\pm},\overline{\mu},\sigma)$ and since ergodicity is preserved under passing to factors we have the result.
\end{proof}

We stress that since $\overline{\mu}$ may not be a Bernoulli measure, the `past' $\cdots a_{-2}a_{-1}a_0$ and `future' $a_1a_2\cdots$ are not independent. However $\overline{\mu}$ is a Gibbs measure and as such has the quasi-Bernoulli property, that there is a constant $C>0$ such that
\[
\frac{1}{C}\,\overline{\mu}[a_{-n}\cdots a_k]\leq \overline{\mu}[a_{-n}\cdots a_0]\,\overline{\mu}[a_1\cdots a_k]\leq C\,\overline{\mu}[a_{-n}\cdots a_k].
\]
Then $\overline{\mu}$ is equivalent to the non-invariant measure $\widetilde{\mu}$ on $\Sigma^{\pm}$ given by
\[
\widetilde{\mu}[a_{-n}\cdots a_k]=\overline{\mu}[a_{-n}\cdots a_0]\,\overline{\mu}[a_1\cdots a_k],
\]
which has independent past and future.

By projecting, $\nu=\overline{\mu}\circ P^{-1}$ is equivalent to the product measure $\widetilde{\nu}=\widetilde{\mu}\circ P^{-1}=\mu_F\times\mu$ on $\mathbb{RP}^1\times\Sigma$. 
This skew product measure is easier to work with, and we use this equivalence in the proof of Lemma \ref{1d}

We now consider how projections of $E$ in different directions are related. When $E$ is a self-affine set with an invariant strong contracting foliation, such as in the case of Bedford-McMullen carpets, projections of $E$ in the strong contracting direction are self-similar sets. In our situation the projections are not self-similar but can be expressed in terms of projections in other directions.

\begin{lemma}
For each $\theta\in\mathbb{RP}^1$ and $i\in\{1,\cdots,m\}$ the map $f_{i,\theta}:[-1,1]\to[-1,1]$ given by
\begin{equation}\label{fith}
f_{i,\theta}=\pi_{\theta}\circ T_i\circ \pi_{\phi_i(\theta)}^{-1}
\end{equation}
is a well-defined affine map   such that $\pi_{\theta}(T_i(E))=f_{i,\theta}(\pi_{\phi_i(\theta)}(E))$.
\end{lemma}
 
\begin{proof}
If $x \in [-1,1]$ then $\pi_{\phi_i(\theta)}^{-1}(x)$ is a line parallel to $\phi_i(\theta)$, so 
$T_i\circ \pi_{\phi_i(\theta)}^{-1}(x)$ is a line parallel to $\theta$, by definition of $\phi_i$. Hence 
$f_{i,\theta}$ is well-defined and is affine since $T_i$ is affine.  Moreover,
$$x\in \pi_{\theta}(T_i(E))\iff  \pi_{\theta}^{-1}(x)\cap T_i(E)\neq \emptyset
\iff T_i^{-1}(\pi_{\theta}^{-1}(x))\cap E\neq \emptyset$$
$$\iff f_{i,\theta}^{-1}(x)=\pi_{\phi_i(\theta)} (T_i^{-1}(\pi_{\theta}^{-1}(x)))\in \pi_{\phi_i(\theta)}(E),$$
so $\pi_{\theta}(T_i(E))=f_{i,\theta}(\pi_{\phi_i(\theta)}(E))$ as $ f_{i,\theta}$ is affine.
\end{proof}

This allows us to deduce that the projections of $E$ form a `self-similar family' in the following sense.
\begin{prop}\label{ProjProp}
With $f_{i,\theta}$ given by \eqref{fith}, $\pi_{\theta}(E)=\bigcup_{i=1}^m f_{i,\theta}(\pi_{\phi_i(\theta)}(E))$ for all $\theta \in \mathbb{RP}^1$.
\end{prop}

\begin{proof}
From \eqref{ifs}, 
\be\label{ssproj}
\pi_{\theta}(E)= \bigcup_{i=1}^{m}\pi_{\theta}(T_i(E))
=\bigcup_{i=1}^m f_{i,\theta}(\pi_{\phi_i(\theta)}(E)).
\ee
\end{proof}

\begin{cor}\label{cordims}
The  dimensions $\dim_H \pi_{\theta}(E)$, ${\underline \dim}_B \pi_{\theta}(E)$ and 
${\overline \dim}_B \pi_{\theta}(E)$ are each constant for $\mu_F$-almost all $\theta\in \mathbb{RP}^1$.
\end{cor}
\begin{proof}
From \eqref{ssproj}, $\dim_H \pi_{\theta}(E)\geq \dim_H \pi_{\phi_i(\theta)}(E)$ for all $\theta\in \mathbb{RP}^1$ and $1 \leq i \leq m$. 
The conclusion follows since $\mu_F$ is ergodic.
\end{proof}

\section{Main proofs }
\setcounter{equation}{0}
\setcounter{theo}{0}
\subsection{Proof of Theorem \ref{HDThm}}

We shall study Hausdorff dimension  by relating the  local dimension of $\mu$ to the local dimension of its images under projection and then estimating the local dimension of these  images. Recall that the {\it local dimension} $\dim_{\rm loc}(\mu,x)$ of $\mu$ at $x$ is given by

\[
\dim_{\rm loc}(\mu,x):=\lim_{\delta\to 0}\frac{\log\mu(B(x,\delta))}{\log\delta}
\]
provided that the limit exists.

The next two lemmas enable us to relate the measures of small balls centred at points in $E$ to the measures of certain slices of the whole of $E$. 
We write $a_n\asymp b_n$ to mean that there is a constant $C$ such that
$a_n/C\leq b_n\leq C a_n$
for all $n\in\mathbb N$ and for any further uniformity specified.

\begin{lemma}\label{equiv}
It is the case that
 \begin{equation}\label{mub}
\mu[a|_n]\,\mu(T_{a|_n}^{-1}(A))\asymp \mu(A)
\end{equation}
for all $a\in \Sigma$, all  $n\geq 0$ and all Borel sets $A\subset E_{a|n}$.
\end{lemma}
\begin{proof}

As $\mu$ is a Gibbs measure it is quasi-Bernoulli, so
\begin{align*}
 \mu(E_{a_1\cdots a_{n+k}}) &=\mu[a_1\cdots a_{n+k}]\\
 &\asymp \mu[a_1\cdots a_n]\, \mu[a_{n+1}\cdots a_{n+k}]\\
 &= \mu[a_1\cdots a_n]\, \mu(E_{a_{n+1}\cdots a_{n+k}})\\
 &= \mu[a_1\cdots a_n]\, \mu(T_{a|_n}^{-1}E_{a_1\cdots a_{n+k}})
\end{align*}
with the implied constant uniform over ${a}, n$ and $k$.
A Borel set $A\subseteq E_{a|n}$ can be approximated arbitrarily closely in measure by a disjoint union of basic sets $E_{a_1\cdots a_{n+k}}$ with $k\geq 0$, each of which is a subset of $E_{a_1\cdots a_n}$. The conclusion follows by summing the measures of these sets and those of their images under $T_{a|_n}^{-1}$.
\end{proof}

To compare local dimensions, we compare the measures of balls with the projected measures of intervals. Note that 
$\alpha_2(a|_n)/\alpha_1(a|_n)  \to 0$ uniformly in $a\in \Sigma$ as $n \to \infty$.  To see this, since the $A_i$ are linear and map the first quadrant strictly into its interior, we can find $\lambda < 1$ so that each $A_i $ contracts angles between lines in the first quadrant by  $\lambda$ or less. Thus the image of the unit square under $A_{i_1}\cdots A_{i_n}$ is a parallelogram with one angle at most $(\pi/2)\lambda^n$, so that the ratio of the width to the diameter of such a parallelogram, which by basic trigonometry is at least  $\alpha_2(a|_n)/\alpha_1(a|_n)$, is at most $(\pi/2)\lambda^n$.

\begin{lemma}\label{asymp}
Let $V$ be a strict subset of ${\rm int}\mathcal Q_2$ such that $\phi_i : V \to  {\rm int}V$ for all $i$. Then there are numbers  $C>0$ and $0< \rho_1< \rho_2$ such that for each $a\in\Sigma, n\in\mathbb N$ and $\theta \in V$,
\begin{align}
&C^{-1}\mu\big(B(\pi(a),  \rho_1\alpha_2(a|_n))\big)\nonumber\\
&\leq \mu\big[a|_n\big]\,
\mu_{\phi_{a_n\cdots a_1}(\theta)}\bigg[\pi_{\phi_{a_n\cdots a_1}(\theta)}(\sigma^n({a}))-\frac{\alpha_2(a|_n)}{\alpha_1(a|_n)}, \pi_{\phi_{a_n\cdots a_1}(\theta)}(\sigma^n({a}))+\frac{\alpha_2(a|_n)}{\alpha_1(a|_n)}\bigg]\nonumber\\
&\leq C\mu\big(B(\pi(a),  \rho_2 \alpha_2(a|_n))\big)\label{equiv1}.
\end{align}
\end{lemma}

\begin{proof} 
Consider the slice  $S$ of the unit disc $D$ given by
$$S=\pi_{\phi_{a_n\cdots a_1}(\theta)}^{-1}\bigg[\pi_{\phi_{a_n\cdots a_1}(\theta)}(\sigma^n({a}))-\frac{\alpha_2(a|_n)}{\alpha_1(a|_n)}, \pi_{\phi_{a_n\cdots a_1}(\theta)}(\sigma^n({a}))+\frac{\alpha_2(a|_n)}{\alpha_1(a|_n)}\bigg],$$
which has side in direction $\phi_{a_n\cdots a_1}(\theta)$. The linear map $T_{a|_n}^{-1}:T_{a|_n}(D)\to D$ maps straight lines in direction $\theta$ to lines 
in direction $\phi_{a_n\cdots a_1}(\theta)$, scaling the spacing between such parallel lines by a factor 
$$ \big(\alpha_1(a|_n)^2\cos^2 \tau+ \alpha_2(a|_n)^2 \sin^2\tau\big)^{-1/2}
$$
where $\tau$ is the angle between  $\theta$ and the minor axis direction $\theta_{a_1\cdots a_n}$ of the ellipse $T_{a|_n}(D)$, using elementary geometry.
Then $T_{a|_n}D$ is an ellipse with major axis of length $2\alpha_1(a|_n)$ and minor axis of length $2\alpha_2(a|_n)$, and $T_{a|_n}S$ is a slice of this ellipse of width 
$$ \frac{\alpha_2(a|_n)}{\alpha_1(a|_n)}\big(\alpha_1(a|_n)^2\cos^2 \tau+ \alpha_2(a|_n)^2 \sin^2\tau\big)^{1/2}$$
and making an angle $\tau$ with the minor axis direction.
Writing $\tau_V$ for $|V|$, that is the angular range of $V$, we have $0\leq |\tau| \leq \tau_V<\frac{\pi}{2}$ and also that the width of the slice $T_{a|_n}S$ is at least
$\alpha_2(a|_n) \cos \tau_V$ and at most $\alpha_2(a|_n)$.  It follows using compactness that there are numbers $0< \rho_1\leq \rho_2$ independent of $a\in\Sigma, n\in\mathbb N$ and $\theta \in V$, such that 
\begin{equation}\label{mes}
B\big(\pi(a),\rho_1 \alpha_2(a|_n)\big) \subset  T_{a|_n}S \subset  B\big(\pi(a),\rho_2 \alpha_2(a|_n)\big).
\end{equation}
We may certainly choose $0< \rho_1< d$ where $d$ is the minimal  separation between the $\{T_iE\}_{i=1}^m$ given by the strong separation condition. This ensures that 
\begin{equation}\label{mes1}
B\big(\pi(a),\rho_1 \alpha_2(a|_n)\big)\cap E \subset T_{a|_n}E,
\end{equation}
 since  if $a,a' \in \Sigma$ with $a|_n =a'|_n$ but  $a|_{n+1} \neq a'|_{n+1}$ then $|\pi(a) - \pi(a')| \geq d \alpha_2(a|_n)$. 

Since $\mu$ is supported by $E$, \eqref{mes} and  \eqref{mes1} give
$$\mu\big(B(\pi(a), \rho_1\alpha_2(a|_n))\big) \leq \mu(T_{a|_n}(E \cap S)) \leq \mu\big(B(\pi(a), \rho_2\alpha_2(a|_n))\big).$$
Taking $A=T_{a|_n}(E \cap S)$ in  Lemma \ref{equiv},  
$$\mu(T_{a|_n}(E \cap S))\asymp \mu[a|_n] \,\mu(E \cap S)= \mu[a|_n]\,
\mu_{\phi_{a_n\cdots a_1}(\theta)}(S),$$ 
so \eqref{equiv1} follows.
\end{proof}

From this lemma we see that, in order to estimate the local dimension of $\mu$ at $a$ we need only estimate the local dimension of the projected measure $\mu_{\phi_{a_n\cdots a_1}\theta}$ at $\pi_{\phi_{a_n\cdots a_1}\theta}(\sigma^n({a}))$.
Thus we work with the approximate local dimensions of the previous lemma, let
\[
d(\theta,{a},n):=\frac{\log\,\mu_{\phi_{a_n\cdots a_1}(\theta)}\!\Big(B(\pi_{\phi_{a_n\cdots a_1}(\theta)}(\sigma^n({a})),\frac{\alpha_2(a|_n)}{\alpha_1(a|_n)})\Big) }{\log\frac{\alpha_2(a|_n)}{\alpha_1(a|_n)}}.
\]

\begin{lemma}\label{1d} For $\nu$-almost every pair $(\theta,a)\in\mathbb{PR}^1\times\Sigma$ and for all $\epsilon>0$, the set
\[G(\theta,a,\epsilon):=\{n\in\mathbb N:\left|d(\theta,{a},n)-1\right|<\epsilon\big\}
\]
satisfies
\[
\lim_{N \to \infty}\frac{1}{N}|G(\theta,a,\epsilon)\cap\{1,\cdots N\}|=1.
\]
\end{lemma}

\begin{proof}
By the assumption of Theorem \ref{HDThm}, for $\mu_F$-almost every $\theta\in \mathbb{RP}^1$ the projected measure $\pi_{\theta}(\mu)$ is absolutely continuous. Then for $\widetilde{\nu}=\mu_F\times\mu$ -almost every pair $(\theta, a)\in \mathbb{RP}^1\times\Sigma$ the Radon-Nikodym derivative of $\mu_{\theta}$ exists and is positive at $\pi_{\theta}(a)$. By the comment on equivalence after the proof of Lemma \ref{ergodic}, this statement also holds for $\nu$-almost every pair $(\theta,a)$. For such pairs $(\theta,a)$
\begin{equation}\label{danlim}
\lim_{r\to 0} \frac{\log\mu_{\theta}(B(\pi_{\theta}(a),r))}{\log r}=1.
\end{equation}
Given $\kappa,\epsilon>0$ let
\[
G_{\kappa,\epsilon}:=\bigg\{(\theta,a):\bigg|\frac{\log\mu_{\theta}(B(\pi_{\theta}(a),r))}{\log r}-1\bigg|<\epsilon \text{ for all } r<\kappa\bigg\} 
\]
Then by \eqref{danlim}, for all $\epsilon>0$
\[
\lim_{\kappa\to 0}\nu(G_{\kappa,\epsilon})=1.
\]
For all $a\in\Sigma$, $\kappa>0$ there exists $N_0\in\mathbb N$ such that
\[
\frac{\alpha_2(a|_n)}{\alpha_1(a|_n)}<\kappa
\]
for all $n>N_0$.

For each $\delta>0$ we may choose $\kappa>0$ such that $\nu(G_{\kappa,\epsilon})>1-\delta.$ Then, recalling that
\[
(P\circ\sigma\circ P^{-1})^n(\theta,a)=(\phi_{a_n\cdots a_1}(\theta), \sigma^n(a))
\]
and that the system $(\mathbb{PR}^1\times\Sigma,\nu,P\circ\sigma\circ P^{-1})$ is ergodic, we see by the ergodic theorem applied to the characteristic function of $G_{\kappa,\epsilon}$ that for $\nu$-almost every $(\theta,a)$,
\begin{align*}
\lim_{N \to \infty}\frac{1}{N}|&G(\theta,a,\epsilon)\cap\{1,\cdots N\}|\\
&=\lim_{N \to \infty}\frac{1}{N}|\{n\in\{1,\cdots N\}: |d(\theta,a,n)-1|<\epsilon\}|\\
&\geq \lim_{N \to \infty}\frac{1}{N}\Big|\big\{n\in\{1,\cdots N\}:\frac{\alpha_2(a|_n)}{\alpha_1(a|_n)}<\kappa,(P\circ\sigma\circ P^{-1})^n(\theta,a)\in G_{\kappa,\epsilon}\big\}\Big|\\
&=\mu(G_{\kappa,\epsilon})>1-\delta.
\end{align*}
Since $\delta$ is arbitrary this completes the proof.

\end{proof}

We can now complete the proof of Theorem \ref{HDThm}. First we have a proposition.
\begin{prop}\label{prop4.4}
For $\mu$-almost every $a\in\Sigma$  and for all $\epsilon>0$ we have
\begin{equation}\label{dens}
\lim_{N\to\infty} \frac{1}{N}\bigg|\{n\in\{1,\cdots,N\}:\bigg|\frac{\log\mu\big(B(\pi({a}),\alpha_2(a|_n))\big)}{\log(\alpha_2(a|_n))}-D(\mu)\bigg|>\epsilon\}\bigg|=0,
\end{equation}
where $D(\mu)$ is the Lyapunov dimension \eqref{lyapdim}.
\end{prop}
\begin{proof}

Firstly recall that for $\mu$-almost every ${a}\in\Sigma$, by the Shannon-McMillan-Breiman theorem, 
\[
\lim_{n\to\infty}-\frac{1}{n}\log\mu[a|_n]= h(\mu),
\]
so by \eqref{lyap}
\begin{equation}\label{ently}
\lim_{n\to\infty}\frac{\log\mu[a|_n]}{\log\alpha_2(a|_n)}=\frac{-h(\mu)}{\lambda_2(\mu)}.
\end{equation}
Then, from the left-hand inequality of Lemma \ref{asymp}, for $\mu$-almost every ${a}\in\Sigma$, $\mu_F$ almost every $\theta \in V$, and for all $n\in G(\theta,a,\epsilon)$,
\begin{align*}
&\frac{\log\mu\big(B(\pi({a}),\rho_1\alpha_2(a|_n))\big)}{\log(\rho_1 \alpha_2(a|_n))}\\
&\quad\leq \frac{\log\big(C\mu[a|_n]\,\mu_{\phi_{a_n\cdots a_1}(\theta)}\!\big(B\big(\pi_{\phi_{a_n\cdots a_1}(\theta)}(\sigma^n({a})),
\alpha_2(a|_n)/\alpha_1(a|_n)\big)\big)}{\log(\rho_1 \alpha_2(a|_n))}\\
&\quad= \frac{\log(C\mu[a|_n])}{\log(\rho_1 \alpha_2(a|_n))}+\frac{\log\mu_{\phi_{a_n\cdots a_1}(\theta)}\!\big(B(\pi_{\phi_{a_n\cdots a_1}(\theta)}(\sigma^n({a})),\alpha_2(a|_n)/\alpha_1(a|_n))\big)}{\log(\rho_1 \alpha_2(a|_n))}\\
&\quad= \frac{\log(C\mu[a|_n])}{\log(\rho_1 \alpha_2(a|_n))}
+d(\theta,a,n)\times\frac{\log\big(\alpha_2(a|_n)/\alpha_1(a|_n)\big)}{\log(\rho_1 \alpha_2(a|_n))}\\
&\quad\leq  \frac{-h(\mu)}{\lambda_2(\mu)}+ (1+\epsilon)\times \frac{\lambda_2(\mu)-\lambda_1(\mu)}{\lambda_2(\mu)}\\
&\quad= \frac{h(\mu)+(1+\epsilon)\lambda_1(\mu)}{-\lambda_2(\mu)}+1+\epsilon\leq (1+\epsilon)D(\mu)+\epsilon.
\end{align*}
Here we have used \eqref{ently}, \eqref{lyap}, and \eqref{lyapdim}. Now by Lemma \ref{1d} we know that for $\nu$ almost every pair $(\theta,a)$ the set $G(\theta,a,\epsilon)$ has density $1$ for all $\epsilon>0$, and so we conclude 
\[
\frac{\log\mu\big(B(\pi({a}),\rho_1\alpha_2(a|_n))\big)}{\log(\rho_1 \alpha_2(a|_n))}\leq (1+\epsilon)D(\mu)+\epsilon
\]
on a set of $n$ of density $1$ for $\mu$-almost every $a$.

A similar reverse inequality with $\rho_2$ instead of $\rho_1$ follows in exactly the same way, using the right-hand inequality in \eqref{equiv1}. In taking upper and lower limits, the values of the constants $\rho_1$ and $\rho_2$ are irrelevant since $\alpha_2(a|_n)\to 0$ no faster than geometrically. Since $\epsilon$ can be chosen arbitrarily small \eqref{dens} follows.

 \end{proof}

To complete the proof of  Theorem \ref{HDThm}, note that the upper and lower limits of 
$\log\mu(B(x,r))/\log r$ as $r \to 0$ are completely determined by any  sequence $r_k \searrow 0$ such that $\log r_{k+1}/\log r_k \to 1$. This is the case taking $r_k = \rho_1 \alpha_2(a|_{n_k}))$ where $n_k$ is any  increasing sequence of positive integers of density 1. It follows from  \eqref{dens} that for $\mu$-almost all $a\in\Sigma$  and all $\epsilon>0$,
$$\bigg|\frac{\log\mu\big(B(\pi({a}),r)\big)}{\log r}-D(\mu)\bigg| < 2\epsilon $$
for all sufficiently small $r$. Hence,  the local dimension of $\mu$ exists and is equal to $D(\mu)$ at $\mu$-almost every $a$.

Since $\mu$ was chosen to be a measure supported by $E$ such that $\dim_A(A_1,\cdots,A_m)= D(\mu)$, we conclude that $\dim_H E \geq \dim_A(A_1,\cdots,A_m)$, with the opposite inequality holding for all self-affine sets. This completes the proof of Theorem \ref{HDThm}.

\subsection{Proof of Theorem \ref{BoxThm}}

The box-counting dimension $\dim_B F$ of a set $F$ is defined in terms of the `box counting numbers' $N(\epsilon, F)$, that is the least number of balls of radius $\epsilon$ that can cover set $F$. We will make use of the well-known fact, see \cite{Fal}, that $N(\epsilon, F)$ is comparable to the number of intervals (in $\mathbb R$) or squares (in $\mathbb R^2$) of the $\epsilon$-grid that overlap $F$.

For $0<\epsilon <1$ let $W(\epsilon)$ be the set of words $a_1\cdots a_n$ for which $\alpha_2(a_1\cdots a_n)<\epsilon$, but $\alpha_2(a_1\cdots a_{n-1})>\epsilon$. The cylinders
\[
\{[a_1\cdots a_n]:a_1\cdots a_n\in W(\epsilon)\}
\]
provide a finite cover of $\Sigma$. We need to estimate  $N(\epsilon, E)$ for small $\epsilon$, which we can relate to the covering numbers of the components  $E_{a_1\cdots a_n}$ by
\be\label{boxnos}
N(\epsilon, E)\leq \sum_{a_1\cdots a_n\in W(\epsilon)} N(\epsilon, E_{a_1\cdots a_n})\leq M N(\epsilon, E)
\ee
for a constant $M$ independent of $\epsilon$ (this follows from an estimate of the areas of the $d/2$-neighbourhoods of the sets $E_{a_1\cdots a_n}$ that overlap a ball of radius $\epsilon$, where $d$ is the minimal separation between the $\{T_i(E)\}$; indeed we can take $M = 24/d^2$).

Let $J$ denote the smallest subinterval of $\mathcal Q_2$ such that $\phi_i(\mathcal Q_2)\subset J$  for each $i$.
The next lemma shows that the box-counting numbers of a component of $E$ change only boundedly under projection in a direction from $J$.

\begin{lemma}\label{lemA}
There exists a constant C such that for all $0<\epsilon\leq1$,  all  $a_1\cdots a_n\in W(\epsilon)$ and all $\theta\in J$,
\be\label{projbound}
\frac{1}{C}N(\epsilon, E_{a_1\cdots a_n})\leq N\big(\epsilon, \pi_{\theta}(E_{a_1\cdots a_n})\big)\leq N(\epsilon,E_{a_1\cdots a_n}).
\ee
\end{lemma}
\begin{proof}
Orthogonal projection from $\mathbb R^2$ onto a line contracts distances, giving the right hand inequality. 

Now note that $J$ lies strictly inside $\mathcal Q_2$ so $\pi_\theta$ is a projection onto a line, $\ell_\theta$ say, of direction uniformly interior to the first quadrant. The set  $E_{a_1\cdots a_n}$ is contained in the ellipse $T_{i_1\cdots i_n}(D)$ which has minor axis of length at most $2\epsilon$ and major axis with direction in the first quadrant. Thus there is an angle $0<\tau<\pi/2$  such that  $\ell_\theta$ that makes an angle at most  $\tau$ with the major axis of $T_{i_1\cdots i_n}(D)$ for all $\theta \in J$ and $a_1\cdots a_n\in W(\epsilon)$. A trigonometric calculation shows that if $B\subset \ell_\theta$ is a covering ball (i.e. interval) of radius $\epsilon$, then 
$\pi_\theta^{-1}(B) \cap T_{i_1\cdots i_n}(D)$ may be covered by at most $C:= 8(\tan\tau + \sec\tau)$ balls of radius $\epsilon$, giving the left-hand inequality.
\end{proof}

We now compare the box-counting numbers of projections of the components $E_{a_1\cdots a_n}$ with those of projections of the set $E$ itself in appropriately chosen directions.

\begin{lemma}\label{lemB}
There is a number $C>0$ such that
\be\label{boxproj2}
N\big(\epsilon,E_{a_1\cdots a_n}\big)
\geq C\ \frac{\alpha_1(a_1\cdots a_n)}{\alpha_2(a_1\cdots a_n)}\mathcal L \big(\pi_{\phi_{a_n\cdots a_1}(\theta)}(E)\big)
\ee
for all $0<\epsilon\leq1$, all $a_1\cdots a_n\in W(\epsilon)$ and all $\theta\in J$.
\end{lemma}
\begin{proof}
The linear map $T_{a_1\cdots a_n}^{-1}:T_{i_1\cdots i_n}(D)\to D$ maps straight lines in direction $\theta$ to lines 
in direction $\phi_{a_n\cdots a_1}(\theta)$, scaling the spacing between such parallel lines by a factor 
\be\label{rhodef}
\rho(a_1\cdots a_n):= \big(\alpha_1(a_1\cdots a_n)^2\cos^2 \tau+ \alpha_2(a_1\cdots a_n)^2 \sin^2\tau\big)^{-1/2}
\ee  
where $\tau\equiv \tau(a_1\cdots a_n) $ is the angle between  $\theta$ and the minor axis direction of the ellipse $T_{i_1\cdots i_n}(D)$, using elementary geometry of the ellipse. It follows that 
\begin{eqnarray*}
N\big(\epsilon,\pi_{\theta}(E_{a_1\cdots a_n})\big)
&\asymp & N\big(\epsilon\rho(a_1\cdots a_n),\pi_{\phi_{a_n\cdots a_1}(\theta)}(E)\big)\\
&\asymp & N\Big(\frac{\alpha_2(a_1\cdots a_n)}{\alpha_1(a_1\cdots a_n)},\pi_{\phi_{a_n\cdots a_1}(\theta)}(E)\Big)
\end{eqnarray*}
with the second equivalence following from \eqref{rhodef}, noting that the $\tau(a_1\cdots a_n) $ are uniformly bounded away from $\pi/2$ and that changing $\epsilon$ by a bounded factor changes $N(\epsilon,F)$ by at most a bounded factor.

Inequality \eqref{boxproj2}  follows noting that $N(\epsilon,F) \geq \lceil \mathcal L (F)/\epsilon \rceil $ for all $F \subset \mathbb R$ and incorporating \eqref{projbound}.
\end{proof}

We can now finish the proof of Theorem \ref{BoxThm}.
\begin{proof}
Throughout this proof `$\asymp$' will mean that the ratio of the two sides is bounded away from $0$ and $\infty$ uniformly in $\epsilon,n$ and $a_1\cdots a_n$.

Let $0< \epsilon<1$. Note that for $1\leq s\leq 2$ and $a_1\cdots a_n \in W(\epsilon)$,
\begin{equation}\label{MuGibbs}
\mu[a_1\cdots a_n]
 \asymp \frac{\alpha_1(a_1\cdots a_n)\alpha_2(a_1\cdots a_n)^{1-s}}{\sum\limits_{b_1\cdots b_n \in W(\epsilon)} \alpha_1(b_1\cdots b_n)\alpha_2(b_1\cdots b_n)^{1-s}}
  \asymp \frac{\alpha_1(a_1\cdots a_n)}{\sum\limits_{b_1\cdots b_n \in W(\epsilon)} \alpha_1(b_1\cdots b_n)}.
\end{equation}
The left-hand equivalence is true because $\mu$ is a Gibbs measure associated to the subadditive potential arising from the cylinder function  $\alpha_1(a_1\cdots a_n)(\alpha_2(a_1\cdots a_n))^{s-1}$ where $1\leq s\leq 2$, see \cite{KaenmakiReeve} or \cite[Definition 2.6]{BaranyRams}. Since $\alpha_2(a_1\cdots a_n)$ is boundedly close to $\epsilon$ for $a_1\cdots a_n\in W(\epsilon)$, we can dispense with the factors $\alpha_2(b_1\cdots b_n)^{s-1}$.

From the hypotheses of the theorem we may choose $w>0$ such that $\mu_F(G) >0$
where $G:= \{\tau \in \mathcal Q_2 : \mathcal L(\pi_\tau E) \geq w\}$. For each $0<\epsilon<1$ 
\begin{eqnarray*}
 \mu_F(G) & \leq&\sum_{\substack{a_1\cdots a_n \in W(\epsilon)\\ \phi_{a_n\cdots a_1}(J) \cap G \neq \emptyset}}
 \mu[a_1\cdots a_n]
 \qquad (\text{taking a covering of $G$ by cylinder sets})\\
& \asymp&\sum_{\substack{a_1\cdots a_n \in W(\epsilon)\\ \phi_{a_n\cdots a_1}(J) \cap G \neq \emptyset}}
\frac{\alpha_1(a_1\cdots a_n)}{\sum\limits_{b_1\cdots b_n \in W(\epsilon)} \alpha_1(b_1\cdots b_n)}
\qquad(\text{by \eqref{MuGibbs}})\\ 
& \asymp&\sum_{\substack{a_1\cdots a_n \in W(\epsilon)\\ \phi_{a_n\cdots a_1}(J) \cap G \neq \emptyset}}
\frac{\alpha_1(a_1\cdots a_n)} {\alpha_2(a_1\cdots a_n)}\  \frac{\epsilon}{\sum\limits_{b_1\cdots b_n \in W(\epsilon)} \alpha_1(b_1\cdots b_n)}
\quad(\text{as $\alpha_2(a_1\cdots a_n)\asymp \epsilon$})\\
& \leq&\sum_{\substack{a_1\cdots a_n \in W(\epsilon)\\ \phi_{a_n\cdots a_1}(J) \cap G \neq \emptyset}}
\frac{1}{C'w}N(\epsilon,E_{a_1\cdots a_n}) \  \frac{\epsilon\, \epsilon^{1-s}}{\sum\limits_{b_1\cdots b_n \in W(\epsilon)} \alpha_1(b_1\cdots b_n)\alpha_2(b_1\cdots b_n)^{1-s}} \\
&& \quad (\text{applying \eqref{boxproj2} with $\theta \in J$ s.t. $\phi_{a_n\cdots a_1}(\theta) \in G$,
 with $\alpha_2(b_1\cdots b_n)\asymp \epsilon$})\\
& \leq&
\frac{M}{C'w}N(\epsilon,E) \ \frac{ \epsilon^{s}}{\sum\limits_{b_1\cdots b_n \in W(\epsilon)} \alpha_1(b_1\cdots b_n)\alpha_2(b_1\cdots b_n)^{1-s}}
\qquad(\text{by \eqref{boxnos}}).
\end{eqnarray*}

We recall that the maps $T_{a_1\cdots a_n}^{-1}$ map lines at angle $\theta$ to lines at angle $\phi_{a_n\cdots a_1}(\theta)$, which accounts for the reversed order of the words $a_n\cdots a_1$ in the above summations. Such sums first arise in Lemma \ref{asymp}.

If $1\leq s <\dim_A E$, it is easily checked, as was shown in \cite{FalSA2}, that there is a number $C_s>0$ such that sum $\sum\limits_{b_1\cdots b_n \in W} \alpha_1(b_1\cdots b_n)\alpha_2(b_1\cdots b_n)^{1-s}\geq C_s$  for all partitions $W$ of $\Sigma$ into cylinders, in particular for $W= W(\epsilon)$.  Thus $N(\epsilon, E) \geq C_s'\epsilon^{-s}$ and so ${\underline \dim}_B E \geq s$ for all $1\leq s <\dim_A E$, from which the conclusion follows.
\end{proof}

\section{Explicit examples of sets with equal Hausdorff and affinity dimensions}\label{Examples}
\setcounter{equation}{0}
\setcounter{theo}{0}

In this final section we present specific classes of self-affine sets which have equal Hausdorff, box-counting and affinity dimensions.

\subsection{Self-affine sets with dimension larger than 1}\label{sabig}

We construct IFSs of affine maps for which $\dim_H \mu+ \dim_H \mu_F>2$ so that box, Hausdorff and affinity dimensions of $E$ are equal by Corollary \ref{Cor2}. This may be the first specific class of affine sets with  Hausdorff dimension larger than one for which the affinity dimension and Hausdorff dimension are known to coincide, apart from examples based on diagonal or upper triangular matrices which have extra structure.

Our example is built out of a large number of contractions $\{T^1_{i,j},T^2_{i,j}\}$, indexed by $1\leq i,j,\leq N$, where the linear parts consist of just two matrices $A_1, A_2$ for which the intervals $\phi_1(\mathcal Q_2)$ and $\phi_2(\mathcal Q_2)$ are disjoint. We use enough contractions to guarantee that the Hausdorff dimension is close to two, while the fact that the Furstenberg measure is supported on a non-overlapping Cantor set allows us to give a lower bound for its Hausdorff dimension.

For angles $0<\tau^- <\tau^+ < \frac{\pi}{2}$ consider the contracting matrix 
\be\label{ataudef}
A(\tau^- ,\tau^+) 
= \frac{1}{2}\left(
\begin{array}{cc}
\cos \tau^-  &   \cos \tau^+  \\
\sin \tau^-  &   \sin \tau^+
\end{array}
\right),
\ee
which maps the unit square into itself and into a cone bounded by half-lines making angles $\tau^-$ and $\tau^+$ with the horizontal axis. The singular values of $A$ are 
\be\label{svs}
\alpha_1:= \frac{1}{2}\big(1+ \cos(\tau^+  - \tau^-)\big)^{1/2}, \quad \alpha_2:=\frac{1}{2}\big(1- \cos(\tau^+  - \tau^-)\big)^{1/2}.
\ee

Now choose angles $0<\tau^-_1 <\tau^+_1 < \tau^-_2 <\tau^+_2< \frac{\pi}{2}$ such that, for convenience, 
$\tau:= \tau^+_1 - \tau^-_1 = \tau^+_2 - \tau^-_2 <\frac{\pi}{4}$. For such $\tau^-_1,\tau^+_1, \tau^-_2 ,\tau^+_2$ define matrices 
$$A_1 = A(\tau^-_1 ,\tau^+_1), \quad A_2 = A(\tau^-_2 ,\tau^+_2).$$
Fix a large integer $N$. For $1\leq i,j \leq N$ define the matrices
$$A_{i,j}^1= \frac{1}{N} A_1, \quad A_{i,j}^2= \frac{1}{N} A_2$$
and affine maps
$$T_{i,j}^1=A_{i,j}^1 + b_{i,j}^1, \quad T_{i,j}^2=A_{i,j}^2+ b_{i,j}^2,$$
where $b_{i,j}^1, b_{i,j}^2$ are translation vectors close to the vector $(i/N,j/N)$ to ensure that each  $T_{i,j}^1$ and $T_{i,j}^2$ map the unit square $[0,1]^2$ onto disjoint parallelograms in the interior of the square $[i/N,j/N]$. 

Let $E_N$ be the attractor of the self-affine set defined by the IFS $\{T_{i,j}^1,T_{i,j}^2 : 1\leq i,j \leq N\}$.   Figure 1 shows a template definining such  $T_{i,j}^1,T_{i,j}^2$ for $N=5$ (where the pararallelograms show the images of the unit square under the affine mappings) along with the corresponding self-affine set.
\begin{figure}[h]\label{fig1}
\begin{center}
\includegraphics[scale=0.35]{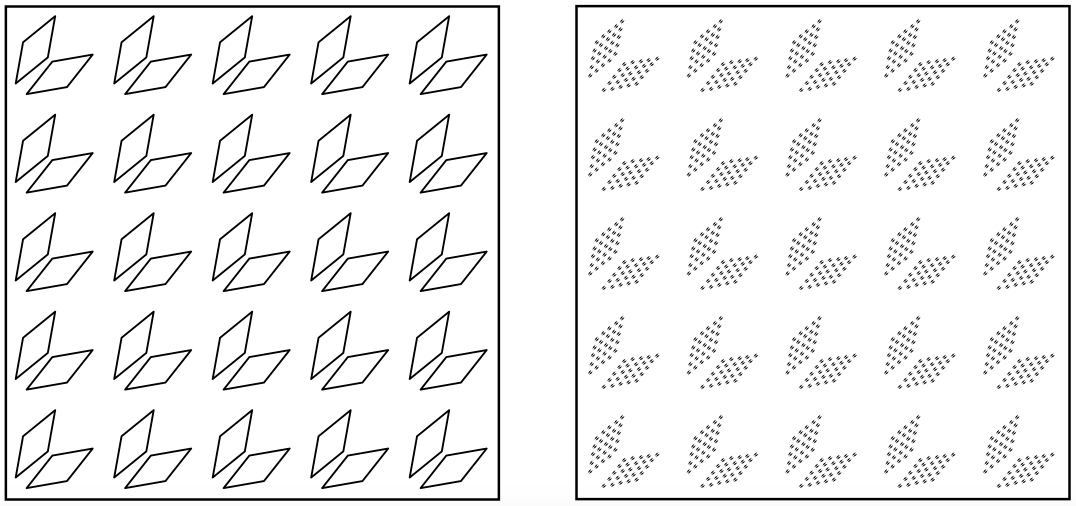}
\caption{Template and self-affine set with equal Hausdorff and affinity dimensions larger than 1}
\end{center}
\end{figure}

From
\eqref{svs}, the singular values of all the $A_{i,j}$ are
$\alpha_1 / N$ and $\alpha_2 / N$. 
\begin{lemma}
Let $\mu$ be the K\"aenm\"aki measure on $E_N$. Then
\begin{equation}
\dim_H \mu \geq \frac{2\log N + \log 2}{\log N - \log \alpha_2}  - \frac{\log(\alpha_1/\alpha_2)}{\log N - \log \alpha_1}.
\end{equation}
In particular, $\dim_H \mu \to 2$ as $N \to \infty$.
\end{lemma}

\begin{proof}
Since $\mu$ is the K\"aenm\"aki measure,
\begin{equation}\label{diml}
D(\mu) =\dim_A E_N \geq \dim_H E_N \geq \frac{\log (2N^2)}{- \log (\alpha_2 / N)},
\end{equation}
where the right-hand inequality follows since the smaller singular values give a lower bound for the Hausdorff dimension of $E$.
Furthermore, 
\begin{equation}\label{dimhmu}
\dim_H \mu \geq \frac{h(\mu)}{-\lambda_2(\mu)}\ =\ D (\mu)  - \frac{\lambda_1(\mu)-\lambda_2(\mu)}{-\lambda_2(\mu)},
\end{equation}
using that  the Hausdorff dimension of a self-affine measure can be bounded below using the smaller Lyapunov exponent, together with the definition of Lyapunov dimension when greater than 1. Relating the Lyapunov exponents to the singular values, 
$$ 
\log\big(\alpha_2 /N\big) \  \leq  \ \lambda_2(\mu)  \ \leq  \ \lambda_1(\mu)  \ \leq  \  \log\big(\alpha_1 /N),
$$
so \eqref{dimhmu} becomes, incorporating \eqref{diml},
\begin{eqnarray}
\dim_H \mu&\geq& \frac{\log (2N^2)}{- \log (\alpha_2 / N)}  - \frac{\log(\alpha_1/N)-\log(\alpha_2/N)}{-\log(\alpha_1/N)}\nonumber\\
&=& \frac{2\log N + \log 2}{\log N - \log \alpha_2}  - \frac{\log(\alpha_1/\alpha_2)}{\log N - \log \alpha_1}.\label{dhe}
\end{eqnarray}
\end{proof}

We now give a lower bound for the dimension of the Furstenberg measure. 
\begin{lemma}
The Furstenberg measure satisfies
\[\dim_H \mu_F \geq \frac{\log 2 - \log \left(\frac{\alpha_1^{2-s(N)}}{\alpha_2^{2-s(N)}}\right)}{\log 8 - \log \tau}\]
where $s(N)=\dim_A(E_N)$, with $s(N)\to 2$ as $N\to\infty$.
\end{lemma}
\begin{proof}
We first show that the K\"aenm\"aki measure $\mu$ is close to being the uniform Bernoulli measure on $\{1,\cdots 2N^2\}^{\mathbb N}$. Since the matrices $A^1_{i,j}, A^2_{i',j'}$ all have the same determinant, 
\[
\alpha_1(a_1\cdots a_n)\alpha_2(a_1\cdots a_n) = (\det(A^1_{1,1}))^n
\]
for any word $a_1\cdots a_n$. In particular, for two words $a_1\cdots a_n$ and $b_1\cdots b_n$,
\[
\frac{\phi_{s(N)}(a_1\cdots a_n)}{\phi_{s(N)}(b_1\cdots b_n)}=\frac{\alpha_1(a_1\cdots a_n)(\alpha_2(a_1\cdots a_n))^{s(N)-1}}{\alpha_1(b_1\cdots b_n)(\alpha_2(b_1\cdots b_n))^{s(N)-1}}=\frac{(\det(A^1_{1,1}))^n}{(\det(A^1_{1,1}))^n}\frac{(\alpha_2(b_1\cdots b_n))^{2-s(N)}}{(\alpha_2(a_1\cdots a_n))^{2-s(N)}}
\]
Using that $\mu$ is a Gibbs measure associated to the subadditive potential $\psi_{s(N)}$, there exists a uniform constant $C$ such that
\[
\frac{\mu[a_1\cdots a_n]}{\mu[b_1\cdots b_n]}\leq C\frac{(\alpha_2(b_1\cdots b_n))^{2-s(N)}}{(\alpha_2(a_1\cdots a_n))^{2-s(N)}} \leq C \left(\frac{\alpha_1^{2-s(N)}}{\alpha_2^{2-s(N)}}\right)^n.
\]
Since there are $(2N^2)^n$ words $a_1\cdots a_n$,
\begin{equation}\label{Kaen}
\mu[a_1\cdots a_n]\leq (2N^2)^{-n} C \left(\frac{\alpha_1^{2-s(N)}}{\alpha_2^{2-s(N)}}\right)^n
\end{equation}
Letting $n\to\infty$ gives, by the Shannon-McMillan-Breiman theorem,
\[
h(\mu)\geq \log (2N^2)- \log \left(\frac{\alpha_1^{2-s(N)}}{\alpha_2^{2-s(N)}}\right).
\]
Finally, since each depth one cylinder in the IFS upon which $\mu_F$ is supported is the image of $N^2$ cylinders for the construction of $E_N$, 
\[
h(\mu_F)= h(\mu)-\log(N^2)\geq \log 2 - \log \left(\frac{\alpha_1^{2-s(N)}}{\alpha_2^{2-s(N)}}\right).
\]

Now writing $\phi \equiv \phi(\tau^- ,\tau^+): \mathcal Q_2 \to \mathcal Q_2$ for the contraction associated with $A$ given by \eqref{phidef}, a routine trigonometric or calculus estimate gives a lower bound for the contraction ratio:
\be\label{contlb}
\big|\phi(\theta_1) - \phi(\theta_2)\big| \geq \frac{|\tau^+  - \tau^-|}{8}|\theta_1 - \theta_2\big| 
\qquad\quad (\theta_1, \theta_2) \in \mathcal Q_2.
\ee
Then the IFS 
$\{\phi(\tau^-_1 ,\tau^+_1),\phi(\tau^-_2 ,\tau^+_2)\}$ on $\mathcal Q_2$ satisfies the strong separation condition. Using standard entropy and Lyapunov exponent arguments gives 
\be\label{furdim}
\dim_H \mu_F \geq \frac{h(\mu_F)}{\log 8 - \log \tau}\geq \frac{\log 2 - \log \left(\frac{\alpha_1^{2-s(N)}}{\alpha_2^{2-s(N)}}\right)}{\log 8 - \log \tau}.
\ee
\end{proof}

\begin{example}\label{largeprop}
In the above construction, let  $0<\tau= \tau^+ - \tau^- < \frac{\pi}{4}$, let $\alpha_1$ and $\alpha_2$ be given by \eqref{svs}, and let  $N$ be  large enough so that
\begin{equation}\label{final}\frac{2\log N + \log 2}{\log N - \log \alpha_2}  - \frac{\log(\alpha_1/\alpha_2)}{\log N - \log \alpha_1} + \frac{\log 2 - \log \left(\frac{\alpha_1^{2-s(N)}}{\alpha_2^{2-s(N)}}\right)}{\log 8 - \log \tau} \ >\ 2.\end{equation}
Then the Hausdorff, box and affinity dimensions of $E_N$ coincide, that is $\dim_H E_N = \dim_B E_N =\dim_A(A_{i,j}^1,A_{i,j}^2: 1\leq i,j\leq N)$.
\end{example}
Note that $s(N)=\dim_A(E_N)$, and so to find explicit $N$ for which (\ref{final}) holds one can use the lower bounds for $\dim_A(E_N)$ given in (\ref{diml}).

\begin{proof}
From  \eqref{furdim} and \eqref{dhe} $\dim_H \mu+ \dim_H \mu_F>2$ so that dimensions coincide by Corollary \ref{Cor2}.
\end{proof}

\subsection{An open set of parameters}
We now show that for small pertubations of the transformations $T^1_{i,j}, T^2_{i,j}$ introduced in Section \ref{sabig}, the Hausdorff dimension of the K\"aenm\"aki measure $\mu$ and the Furstenberg measure still satisfy $\dim_H \mu+\dim_H \mu_F>2$. This gives an open set of affine transformations for which the Hausdorff, box and affinity dimensions of the attractor  coincide. 

First we note that our lower bound for $\dim_H \mu$ in \eqref{dhe} of Section \ref{sabig} was based solely on the number $N$, where there are $2N^2$ contractions, together with the singular values of the matrices $A^1_{i,j}, A^2_{i,j}$ and the Lyapunov dimension of $\mu$. The singular values of a matrix are continuous in the entries of the matrix, as is the Lyapunov dimension of $\mu$ \cite{FengShmerkin}, so our lower bound for $\dim_H \mu$ is continuous under small perturbations.

The lower bound for $\dim_H \mu_F$ is slightly more subtle, since the $N^2$ cylinders corresponding to each of $T^1_{i,j}$, $T^2_{i,j}$ in the IFS generating $\mu_F$ need no longer overlap exactly. However it is still the case that, for small perturbations of the original system, at most $(N^2)^n$ cylinders of depth $n$ can cover a single $\theta\in\mathbb{PR}^1$. Furthermore, the K\"aenm\"aki measure for our perturbed system is still close to satisfying inequality (\ref{Kaen}), and so our lower bound (expressed via entropy) for the mass of depth $n$ cylinders covering a single $\theta\in\mathbb{PR}^1$ still holds. Finally, since the parameters defining the contraction ratios for the IFS defining $\mu_F$ are continuous in the perturbation, we get a lower bound for $\dim_H \mu_F$ which varies continuously as the system is perturbed.

Thus for small perturbations of the IFS of Section \ref{sabig}, the inequality $\dim_H \mu +\dim_H \mu_F>2$ remains valid , so by Corollary \ref{Cor2} the Hausdorff, box and affinity dimension of the corresponding self-affine sets coincide.

\subsection{Self-affine sets with dimension less than 1}

Finally, we construct a  family of self-affine sets of dimension less than 1, each contained in a Lipschitz  curve and with equal Hausdorff and affinity dimensions. The family is defined by a simple condition on the associated mappings $\phi_i$ on $\mathcal Q_2$ given by \eqref{phidef}, though it does not  directly depend on our main theorems. This condition, which gives open sets of affine transformations for which the Hausdorff and affinity dimensions are equal, is very different from that of Heuter and Lalley \cite{HueterLalley} who presented a different family of such sets; see also \cite{PV} for a discussion of the `size' of their parameter family. 

 We first need a linear algebra lemma on the comparability of eigenvalues and singular values which is probably in the literature, though we have been unable to find a reference.

\begin{lemma}\label{lemeig}
For all $0<\epsilon<1 $ there is a number $c>0$, depending only on $\epsilon $, such that if $A$ is a $2\times 2$ matrix with real eigenvalues $|\lambda_1|\geq |\lambda_2|>0$ and  corresponding normalised eigenvectors $e_1, e_2$ such that $|e_1 \cdot e_2| < 1-\epsilon$, then the singular values $\alpha_1 \geq \alpha_2 >0$ satisfy
\begin{equation}\label{eigsv}
c^{-1} |\lambda_i | \leq \alpha_i \leq c |\lambda_i |\qquad (i=1,2).
\end{equation}
\end{lemma}

\begin{proof}
We may diagonalise $A$ so that $P^{-1}AP = \mbox{diag}(\lambda_1,\lambda_2)$ where $P$ has columns given by the vectors $e_1$ and $ e_2$. By the submultiplicativity of the Euclidean norm  $\|\ \|$,
\begin{equation}\label{est}
\|A\| \leq \|P\| \|P^{-1}\|\,  |\lambda_1| \quad \mbox{  and } \quad\|\lambda_1| \leq  \|P\| \|P^{-1}\| \|A\|.
\end{equation}
By direct calculation $\det P^T P = 1- |e_1 \cdot e_2|^2$, so with $\alpha_1(P)\geq \alpha_2(P)$ as the singular values of $P$,
$$\|P\| \|P^{-1}\|=\frac{\alpha_1(P)}{\alpha_2(P)} = \frac{\alpha_1(P)^2}{\alpha_1(P)\alpha_2(P)}=  \frac{\|P\|^2}{\det  P^T P}
\leq \frac{4}{1- |e_1 \cdot e_2|^2},$$
since all entries of $P$ are at most $1$ in absolute value. (In numerical analysis $\|P\| \|P^{-1}\|$ is referred to as the {\it condition number} of $P$).
Since  $\|A\| =\alpha_1$, \eqref{eigsv} follows from \eqref{est} in the case of $i=1$. The result follows for $i=2$  by applying the conclusion for $i=1$ to the inverse $A^{-1}$ which has larger eigenvalue $1/\lambda_2$ and larger singular value $1/\alpha_2$. 
\end{proof}

As before, let $J$ be the minimal closed interval in  $\mathcal Q_2$ such that $\phi_i(J) \subset J$ for all $i = 1, \ldots,m$. Assuming the strong separation condition, let $S\subset \mathbb{RP}^1$ be the closed set of directions  realised by pairs of points in distinct components of $T_i(E)$, that is $S =\{\widehat{x-y}: x\in T_i(E), y\in T_j(E) \mbox{ where } i \neq j\}$, where $\widehat{w} \in \mathbb{RP}^1$ denotes the unit vector in the direction of the vector $w$.

\begin{prop}\label{smallprop}
With notation as above, if  $J$ and $S$  are disjoint then the self-affine set $E$ is contained in a Lipschitz curve and $\dim_H E = \dim_B E =\dim_A(A_1,\cdots,A_m)$.
\end{prop}

\begin{proof}
Since each $\phi_i$ maps $J$ into itself, each $\phi_i^{-1}$ maps $\mathbb{RP}^1\setminus J$ into itself. If $x,y$ are distinct points of $E$  we may write $x= T_{a_1\cdots a_n}x_0$ and $y= T_{a_1\cdots a_n}y_0$ for some $n$, where $x_0 \in T_i(E) $ and $y_0\in T_j(E)$ with $i \neq j$. In particular,  $\widehat{x_0 - y_0} \in S \subseteq
\mathbb{RP}^1\setminus J$, so that 
$$\widehat{x - y} = \widehat{A_{a_1\cdots a_n}(x_0 - y_0)} = \phi^{-1}_{a_1}\cdots \phi^{-1}_{a_n}(\widehat{x_0 - y_0})\in \mathbb{RP}^1\setminus J,$$
noting that each $\phi_i^{-1}$ is simply the action of the $A_i$ on the direction of vectors. Let $v$ be a unit vector in $J$; since $J$ and $S$ are closed and disjoint, the angle between $v$ and all vectors $x-y$ with $x,y \in E$ is bounded away from $0$, so that $E$ is contained in the graph of a Lipschitz function above an axis perpendicular to $v$.

Again with $x= T_{a_1\cdots a_n}x_0$ and $y= T_{a_1\cdots a_n}y_0$ as above, let $A_{a_1\cdots a_n}$ have  eigenvalues $|\lambda_1|\geq |\lambda_2|>0$ with corresponding normalised eigenvectors $e_1, e_2$. Then $e_1 \in 
\mathcal Q_1$ and $e_2 \in J$ so $|e_1\cdot e_2|<1-\epsilon_1$ for some $\epsilon_1 >0$ independent of $a_1\cdots a_n$. Furthermore, $\widehat{x_0-y_0} \in S$ makes an angle at least $\epsilon_2$ with $e_2$,where $\epsilon_2>0$ is the minimum angle between $S$ and $J$. It follows that we may write
$$ x_0 - y_0 = r_1 e_1 + r_2 e_2$$
where $r_1$ and $r_2$ are scalars such that $|r_1|\geq b_1|r_2|$, so also $|r_1|\geq b_2|x_0 - y_0|$, where $b_1, b_2>0$ depend only on $\epsilon_1$ and $\epsilon_2$. Then
$$x-y = T_{a_1\cdots a_n}(x_0-y_0)= A_{a_1\cdots a_n}(x_0-y_0)
= A_{a_1\cdots a_n}(r_1 e_1 + r_2 e_2) = r_1 \lambda_1e_1 + r_2 \lambda_2 e_2.$$
Using Lemma \ref{lemeig}, 
\begin{align}|x-&y| \geq  |r_1| (|\lambda_1| - |\lambda_2|/b_1) \geq b_3 |r_1| \big(\alpha_1(a_1\cdots a_n)- b_4 \alpha_2(a_1\cdots a_n)\big)\nonumber\\
&\geq b_3 |r_1| \alpha_1(a_1\cdots a_n) /2 \geq b_5|x_0 - y_0|\alpha_1(a_1\cdots a_n) \geq b_5 d\alpha_1(a_1\cdots a_n) \label{distal}
\end{align}
where $d>0$ is the minimum separation of the  $T_i(E)$, provided that $n \geq n_0$ is sufficiently large, where the $b_i$ and $n_0$ do not depend on $x,y$ or $(a_1\cdots a_n)$.

Define a metric $\rho$ on $E$ by $\rho(x,y) = \min\{\alpha_1(a_1\cdots a_n) : x, y\in  E_{a_1\cdots a_n}\}$ for $x,y \in E,\ x\neq y$. Then $\rho$ is well-defined and is an ultrametric by virtue of the tree structure of $\Sigma$. Moreover, it follows from \eqref{distal} that the identity $i: (E, |\cdot|) \to (E, \rho)$ is a Lipschitz (in fact a bi-Lipschitz) mapping . 

Let $0<s< \dim_A (A_1,\ldots,A_m)<1$. Suppose that  $E \subset \bigcup_ { (a_1\cdots a_n) \in {\mathcal S} }  E_{a_1\cdots a_n} $ for some ${\mathcal S}\subset  \bigcup_{k=0}^\infty \{1,2,\ldots,m\}^k $, that is  the cylinders defined by ${\mathcal S}$ cover $\Sigma$.  It follows from the submultiplicativity of the $\alpha_1$ and the definition of $\dim_A E$ that  $ \sum_ { (a_1\cdots a_n) \in {\mathcal S} }\alpha_1(a_1\cdots a_n)^s = \infty$, see, for example,  \cite[Proposition 4.1]{Falconer88}. Thus the Hausdorff dimension of $E$ with respect to the metric $\rho$ is at least $s$, so as $i: (E, |\cdot|) \to (E, \rho)$ is Lipschitz the same is true with respect to the usual metric   $|\cdot|$. 
This is true for all $0<s< \dim_A(A_1,\ldots,A_m)$, so $\dim_H E\geq \dim_A(A_1,\ldots,A_m)$, and the opposite inequality holds for all self-affine sets, see \cite{Falconer88}.

\end{proof}

\begin{figure}[h]\label{fig2}
\begin{center}
\includegraphics[scale=0.35]{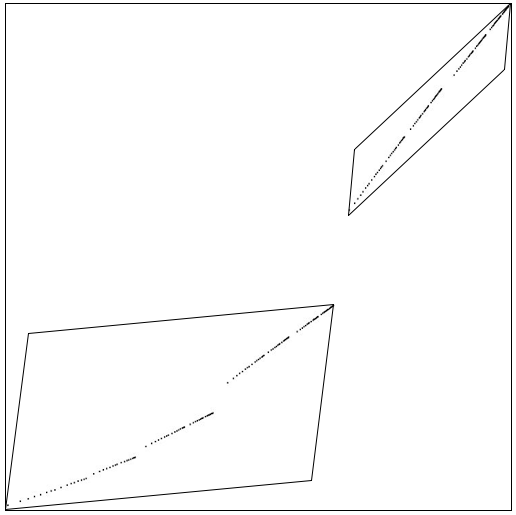}
\quad
\includegraphics[scale=0.35]{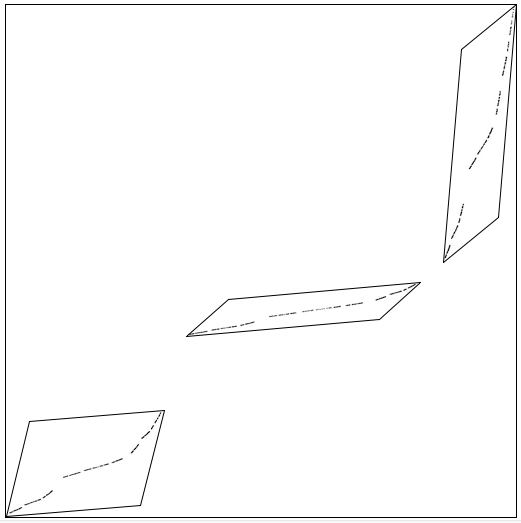}
\caption{Templates for self-affine sets with equal Hausdorff and affinity dimensions less than 1 }
\end{center}
\end{figure}

It is easy to specify sets of affine transformations satisfying  Proposition \ref{smallprop}, templates for two examples are shown in  Figure 2. 
\begin{example}\label{egsmall}
For $i=1,2$ let 
$P_i = 
\left(
\begin{array}{cc}
1  &   -b_i   \\
  c_i&   1
\end{array}
\right),
$ 
where $b_i,c_i>0$, and let $\lambda_i > \mu_i >0$.
Define an iterated function system by
\begin{equation}\label{smallifs}
T_1 =P_1  \left(\begin{array}{cc}
\lambda_1  &   0  \\ 0&   \mu_1
\end{array}\right)
P_1^{-1}, \quad T_2= P_2 \left(\begin{array}{cc}
\lambda_2  &   0  \\ 0&   \mu_2
\end{array}\right)P_2^{-1} 
+ \binom{a}{b}.
\end{equation}
If
\begin{equation}\label{conds}
 \lambda_1\big(1+\max\{b_1,c_1\}\big)^2 < a,b < 1-\lambda_2\big(1+\max\{b_2,c_2\}\big)^2
 \end{equation}
 then $E$ is contained in a Lipschitz curve within the unit square  and $\dim_H E = \dim_B E =\dim_A(A_1,A_2)$
 where $A_i=P_i\,{\rm diag}(\lambda_i, \mu_i)P_i^{-1}$ are the linear parts of the affine maps $T_i$.
\end{example}

\begin{proof}
Note that the matrices $P_i\,{\rm diag}(\lambda_i, \mu_i)P_i^{-1}$ map the first quadrant into itself without orientation reversal. Condition \eqref{conds} ensures that the $T_i$ map the unit square into itself with the projections of $T_1( [0,1]^2)$ and $T_2 ([0,1]^2)$ onto both horizontal and vertical axes disjoin, noting that
$\| P_i\,{\rm diag}(\lambda_i, \mu_i)P_i^{-1}\|_\infty \leq \lambda_i\big(1+\max\{b_i,c_i\}\big)^2$. It follows that, with the notation of Proposition \ref{smallprop}, $S\subset \mathcal Q_1$, but $J \subset \mathcal Q_2$ (in fact $J$ is the interval bounded by the directions of the eigenvectors of $P_1$ and $P_2$ corresponding to the smaller eigenvalues). In particular $J$ and $S$ are disjoint and the conditions of Proposition \ref{smallprop} are satisfied.
\end{proof}

Example \ref{egsmall} provides an open set of IFSs with respect to the natural parameterization for which the attractor $E$ has equal Hausdorff dimension and affinity dimension. To see this, note that a matrix $A_i$ that maps the first quadrant into itself can always be diagonalised using a matrix $P_i$  of the form stated, and also nothing is lost by setting the translation component of $T_1$ to 0 (since adding a constant translation to both maps just shifts the attractor correspondingly).

\bibliographystyle{plain} 
\bibliography{Lyapunov.bib}
\end{document}